\documentclass[10 pt,english]{article}
\usepackage[activeacute]{babel}
\usepackage{amsfonts} 
\usepackage{mathrsfs} 
\usepackage{pstricks-add}
\usepackage{amssymb, amsmath, amsfonts}
\usepackage{makeidx}
\usepackage{color}
\usepackage{pstricks, pst-node}
\usepackage{graphics,graphicx,graphpap}
\usepackage[ansinew]{inputenc}
\usepackage{fancyhdr} 
\usepackage{amsthm}
\usepackage{epsfig,graphicx}
\usepackage{amsmath,amsfonts,verbatim}

\if@twoside \oddsidemargin 5pt \evensidemargin 5pt \marginparwidth 20pt
 \else \oddsidemargin 5pt \evensidemargin 5pt \marginparwidth 20pt \fi

\newcommand{\ZZ}{\mathbb{Z}}

\newcommand{\X}{\mathcal{X}}

\newcommand{\B}{{\cal B}}
\renewcommand{\O}{{\cal O}}

\newcommand{\aut}{\hbox{{\rm Aut}}}
\newcommand{\HAT}{\hbox{{\rm HAT}}}
\newcommand{\GHAT}{\hbox{{\rm GHAT}}}

\newcommand{\Circ}{\mathrm{Circ}}

\newcommand{\G}{\Gamma}

\renewcommand{\L}{\ell}
\newcommand{\rad}{\mathrm{rad}}
\newcommand{\att}{\mathrm{att}}
\newcommand{\jum}{\mathrm{jum}}
\newcommand{\Alt}{\mathrm{Alt}}

\newtheorem{theorem}{Theorem}[section]
\newtheorem{proposition}[theorem]{Proposition}
\newtheorem{corollary}[theorem]{Corollary}
\newtheorem{lemma}[theorem]{Lemma}
\newtheorem{question}[theorem]{Question}
\newtheorem{problem}[theorem]{Problem}

\theoremstyle{definition}
\newtheorem{construction}[theorem]{Construction}

\begin{document}
\begin{center}
\Large{\textbf{New structural results on tetravalent half-arc-transitive graphs}} \\ [+4ex]
Alejandra Ramos Rivera{\small$^{a,}$\footnotemark $^{,*}$}, \
Primo\v z \v Sparl{\small$^{a, b, c, }$\footnotemark}  
\\ [+2ex]
{\it \small 
$^a$University of Primorska, IAM, Muzejski trg 2, 6000 Koper, Slovenia\\
$^b$University of Ljubljana, Faculty of Education, Kardeljeva plo\v s\v cad 16, 1000 Ljubljana, Slovenia\\
$^c$IMFM, Jadranska 19, 1000 Ljubljana, Slovenia}
\end{center}

\addtocounter{footnote}{-1}
\footnotetext{Supported in part by the Slovenian Research Agency (research program P1-0285 and Young Researchers Grant no. 37541).}
\addtocounter{footnote}{1} 
\footnotetext{Supported in part by the Slovenian Research Agency (research program P1-0285 and research projects N1-0038, J1-6720, J1-7051).

Email addresses: 
alejandra.rivera@iam.upr.si (Alejandra Ramos Rivera),
primoz.sparl@pef.uni-lj.si (Primo\v z \v Sparl).

~*Corresponding author  }

\hrule
\begin{abstract}

Tetravalent graphs admitting a half-arc-transitive subgroup of automorphisms, that is a subgroup acting transitively on its vertices and its edges but not on its arcs, are investigated. One of the most fruitful approaches for the study of structural properties of such graphs is the well known paradigm of alternating cycles and their intersections which was introduced by Maru\v si\v c 20 years ago.  

In this paper a new parameter for such graphs, giving a further insight into their structure, is introduced. Various properties of this parameter are given and the parameter is completely determined for the tightly attached examples in which any two non-disjoint alternating cycles meet in half of their vertices. Moreover, the obtained results are used to establish a link between two frameworks for a possible classification of all tetravalent graphs admitting a half-arc-transitive subgroup of automorphisms, the one proposed by Maru\v si\v c and Praeger in 1999, and the much more recent one proposed by Al-bar, Al-kenai, Muthana, Praeger and Spiga which is based on the normal quotients method. 

New results on the graph of alternating cycles of a tetravalent graph admitting a half-arc-transitive subgroup of automorphisms are obtained. A considerable step towards the complete answer to the question of whether the attachment number necessarily divides the radius in tetravalent half-arc-transitive graphs is made.  
\end{abstract}
\hrule

\begin{quotation}
\end{quotation}

\section{Introduction}

Throughout this paper all graphs are assumed to be simple, finite, connected and undirected (but with an implicit orientation of the edges when appropriate). Let $\G$ be a graph and let $V(\G)$, $E(\G)$ and $A(\G)$ be the sets of its vertices, edges and arcs, respectively, where an {\em arc} of $\G$ is an ordered pair of vertices $(u,v)$ such that $uv \in E(\G)$ (each edge $uv$ thus gives rise to two arcs $(u,v)$ and $(v,u)$). For a subgroup $G \leq \aut(\G)$ the graph $\G$ is said to be {\em $G$-vertex-transitive}, {\em $G$-edge-transitive} or {\em $G$-arc-transitive} if $G$ acts transitively on $V(\G)$, $E(\G)$ or $A(\G)$, respectively. If $G$ acts vertex- and edge-transitively but not arc-transitively, then $\G$ is said to be {\em $G$-half-arc-transitive}. In this case we also say that $\G$ admits a half-arc-transitive group of automorphisms. If $G=\aut(\G)$ we omit the prefix $\aut(\G)$ in the above definitions. 

The first result concerning graphs admitting a half-arc-transitive action was given by Tutte~\cite{Tutte}, who proved that the valency of such graphs must be even. Since any connected 2-valent graph is a cycle, the smallest interesting valency for the study of graphs admitting a half-arc-transitive group of automorphisms is four. It is thus not surprising that the majority of papers on such graphs deal with the tetravalent ones. A lot of work has been done to reach the long-term goal of obtaining a complete classification of all tetravalent graphs admitting a half-arc-transitive action and/or a complete classification of all half-arc-transitive tetravalent graphs. Despite the fact that numerous papers on the topic have been published in the last half a century the complete classification appears to be a very difficult problem and is currently still out of reach. As a result various restricted subproblems have been considered and different general approaches to the study of tetravalent graphs admitting a half-arc-transitive group of automorphisms have been proposed. For instance, the graphs of restricted orders such as $p^3$, $p^4$, $pq$, $3p$, $4p$, $2pq$, etc., where $p$ and $q$ are prime numbers, have been classified, some even for all valencies (see for instance \cite{AX94, Dob06, FenKwaZho11, FenKwaXuZho08, KutMarSpaWanXu13, Xu92}). The vertex-stabilizers in tetravalent graphs admitting a half-arc-transitive group of automorphisms and the connection of such graphs to maps have also been studied (see for instance \cite{ConPotSpa15, Mar05, MarNed98, MarNed01}). Recently, Poto\v cnik, Spiga and Verret constructed a census~\cite{PotSpiVer15} of all tetravalent graphs admitting a half-arc-transitive group action up to order 1000. Since we will be referring to some of the graphs from the census we mention that they have names of the form $\HAT[n,i]$ or $\GHAT[n,i]$, where $n$ is the order of the graph and the prefix G indicates that the full automorphism group of the graph acts arc-transitively.   
  
One of the most fruitful and general approaches to the study of structural properties of tetravalent graphs admitting a half-arc-transitive action was started 20 years ago by Maru\v{s}i\v{c}~\cite{Mar98}. It is based on the investigation of certain cycles called \textit{alternating cycles}. We give a description of the main ideas but we refer to \cite{Mar98} for details. Let $\G$ be a tetravalent $G$-half-arc-transitive graph for some $G \leq \aut(\G)$. It is easy to see that  the action of $G$ on $A(\G)$ has two paired orbits. Let $\mathcal{O}_G$ be one of them. Then for each edge of $\G$ the orbit $\mathcal{O}_G$ contains exactly one of the two arcs corresponding to this edge, and so $\mathcal{O}_G$ gives rise to an orientation of the edges of $\G$, preserved by the action of $G$. We say that the orientation $\mathcal{O}_G$ is {\em $G$-induced}. A cycle of $\G$ is a {\em $G$-alternating cycle} if each pair of its consecutive edges have opposite orientations in $\mathcal{O}_G$, that is, if traversing the cycle we alternate between traveling with and against the orientation of the edges from $\mathcal{O}_G$. It turns out that all $G$-alternating cycles have the same length, half of which is called the {\em $G$-radius} of $\G$ and is denoted by $\rad_G(\G)$. Moreover, any two $G$-alternating cycles with non-empty intersection share the same number of vertices. This number, denoted by $\att_G(\G)$, is called the {\em $G$-attachment number}, while the intersections themselves are called {\em $G$-attachment sets}. The study of alternating cycles and their intersections has been the topic of several papers with most of the attention given to the {\em loosely}, {\em antipodally} and {\em tightly $G$-attached} graphs in which $\att_G(\G)$ attains one of the extremal values $1$, $2$ or $\rad_G(\G)$, respectively (see for instance \cite{Mar98, MarWall00, PotSpa17}).

The importance of these three special cases is based on the results of \cite{MarPra99}, where it was proved that each tetravalent $G$-half-arc-transitive graph $\G$, where $G\leq \aut(\G)$, is either tightly $G$-attached or admits an imprimitivity block system for $G$ such that the corresponding quotient graph $\G_\B$ is loosely or antipodally attached (see Section~\ref{sec:kernel} for details). The tightly $G$-attached graphs have been classified  (\cite{Mar98, MarPra99,Spa08,Wil04}), while in the remaining cases there is still a lot of work to be done. We would like to point out however, that in~\cite{MarPra99} it was not known whether the mentioned imprimitivity block system for $G$, giving rise to $\G_\B$, can be obtained as the set of orbits of a normal subgroup of $G$ or not. 

This question is of great importance. Namely, in \cite{AlAlMuPrSp2016} a new framework for a possible classification of tetravalent graphs admitting a half-arc-transitive group of automorphisms was proposed. It is based on the well known method of taking normal quotients, which has already led to important results in the study of graphs possessing a considerable degree of symmetry. For instance, its use in the context of $s$-arc-transitive graphs was initiated by Praeger in 1993~\cite{Pra93}. The main idea in our setting is that whenever a half-arc-transitive subgroup $G \leq \aut(\G)$ for a tetravalent graph $\G$ has a normal subgroup $N$ with at least three orbits, the quotient graph with respect to the orbits of $N$ is again a tetravalent graph (provided that it is not a cycle with `doubled edges') admitting a half-arc-transitive subgroup of automorphisms (a quotient group of $G$). One thus aims to classify all `minimal' examples (not admitting such normal subgroups) and to understand how the remaining graphs can be reconstructed from the minimal ones (see~\cite{AlAlMuPrSp2016} for details). Recently, some results of this kind have been obtained in \cite{AlAlMuPr17, 2AlAlMuPr17}. 
\medskip

The results of the present paper enable us to link the above mentioned approaches from~\cite{Mar98, MarPra99} and \cite{AlAlMuPrSp2016}. We first introduce a new parameter for tetravalent graphs admitting a half-arc-transitive action called the \textit{alternating jump} (see Section~\ref{sec:jump} for the definition), which gives a more detailed insight into the structure of such graphs when compared to the one given by simply considering their radius and attachment number. The obtained results enable us to prove that the imprimitivity block system giving rise to $\G_\B$ from~\cite{MarPra99} is in fact obtained by orbits of a normal cyclic subgroup (see Theorem~\ref{the:quotient}), which thus links the quotients of~\cite{MarPra99} to those of~\cite{AlAlMuPrSp2016}.

There is another possibility for the study of tetravalent graphs $\G$ admitting a half-arc-transitive group of automorphisms $G \leq \aut(\G)$, still based on the $G$-alternating cycles. Instead of considering the quotient graph $\G_\B$ one can study the graph $\Alt_G(\G)$ of $G$-alternating cycles introduced in~\cite{PotSpa17} (see Section~\ref{sec:kernel} for the definition). Employing the properties of the alternating jump parameter we determine the kernel of the natural action of $G$ on $\Alt_G(\G)$ (see Theorem~\ref{the:kernel}) and show that this kernel is closely related to the kernel of the natural action of $G$ on $\G_\B$ as well as to the kernel of its action on the set of all $G$-attachment sets of $\G$ (see  Theorem~\ref{the:allkernels}). The group $G$ induces a natural vertex- and edge-transitive action on both $\G_\B$ and $\Alt_G(\G)$.  We consider the question of when the graphs $\G_\B$ and $\Alt_G(\G)$ are half-arc-transitive and when they are arc-transitive, and we indicate some connections between these two graphs.

Finally, we address the question posed in~\cite{PotSpa17}, of whether the attachment number of a tetravalent half-arc-transitive graph (that is a graph in which the full automorphism group is half-arc-transitive) necessarily divides its radius. Using the results on the alternating jump parameter we prove that the answer to this question is in the affirmative, except possibly for the cases when $\att(\G) = 4$ or the alternating jump parameter is equal to $\att(\G)/2 - 1$ (see Theorem~\ref{the:andivr}).
 
\section{Preliminaries}
\label{pre}

As explained in the Introduction the tightly attached graphs are of great importance in the study of tetravalent graphs admitting a half-arc-transitive group of automorphism and have been classified in~\cite{Mar98,MarPra99,Spa08,Wil04}. Since we will be referring to these graphs quite often in this paper we include their definitions for self-completeness. The definition is given in two parts depending on the parity of the radius of the graphs, where we follow~\cite{Spa08}.

\begin{construction}[\cite{Spa08}]
For each $m \geq 3$, $r \geq 3$ odd, $q \in \ZZ_r^*$, where $q^m=\pm 1$, let $\X_o(m,r;q)$ be the graph with vertex set $V=\{u^j_i \colon i \in \ZZ_m, j \in \ZZ_r\}$ and edges defined by the following adjacencies: 
$$
u^j_i \sim \left\{\begin{array}{lcl}
u^{j\pm r^i}_{i+1} & ; &  i \in \ZZ_m ,j\in \mathbb{Z}_r.
\end{array}\right.
$$
\end{construction}

\begin{construction}[\cite{Spa08}]
Let $m\geq 4$, $r \geq 4$ and for each $q \in \ZZ_r^*$, $t \in \ZZ_n$ satisfying 
$$
q^m=1\mbox{,            } t(q-1)=0\mbox{       and         }  1+q+\cdots+q^{m-1}+2t=0,
$$ 
let $\X_e(m,r;q,t)$ be the graph with vertex set $V=\{u^j_i \colon i \in \ZZ_m, j \in \ZZ_r\}$ and edges defined by the following adjacencies:
$$
u^j_i \sim \left\{\begin{array}{lcl}
u^{j}_{i+1}, u^{j+r^i}_{i+1} & ; &  i \in \ZZ_m \backslash \{m-1\} ,j\in \mathbb{Z}_r\\
u^{j+t}_{0}, u^{j+r^{m-1}+t}_{0}& ; &  i = m-1 , j\in \mathbb{Z}_r.

\end{array}\right.
$$
\end{construction}
We remark that when these graphs where introduced in~\cite{Spa08}, different letters for the second and third parameter of the graphs were used. However, since the second parameter is in fact the radius of these graphs and we show that the third parameter coincides with the new parameter from Section~\ref{sec:jump} which we consistently denote by $q$ (see Section~\ref{sec:GTA}), we decided to denote the second parameter by $r$ and third by $q$, not to confuse the reader. 

Let $\G$ be a tetravalent graph admitting a half-arc-transitive group of automorphisms $G \leq \aut(\G)$. As explained in the Introduction the action of $G$ induces two paired orientations of the edges of $\G$. Throughout the paper we will be dealing with situations when we fix one of these two orientations. Let $u, v$ be a pair of adjacent vertices of $\G$. In the case that the edge $uv$ is oriented from $u$ to $v$ we say that $u$ is the {\em tail} and that $v$ is the {\em head} of the edge $uv$ (and of the arc $(u,v)$). Let us point out a useful fact that we use recurrently in this paper. Since the orientation of the edges of $\G$ is given by the action of $G$, the elements of $G$ permute the $G$-alternating cycles, and so $G$ induces a natural action on the set of all $G$-alternating cycles as well as on the set of the $G$-attachment sets. 

A half-arc-transitive subgroup $G$ of automorphisms of a tetravalent graph $\G$ may have large vertex-stabilizers (see for instance~\cite{Mar05, MarNed01}). However, by the following result of \cite{MarPra99} which we will be using in this paper, this cannot occur if the $G$-attachment number $\att_G(\G)$ is at least $3$. 

\begin{proposition}\cite[Lemma~3.5.]{MarPra99}
\label{pro:stab}
Let $\G$ be a tetravalent graph admitting a half-arc-transitive group of automorphisms $G \leq \aut(\G)$. If $\att_G(\G) \geq 3$ then $G_v \cong \ZZ_2$ for all $v \in V(\G)$.
\end{proposition}

For the sake of completeness we also include the definition of a circulant graph. Let $n$ be an integer and let $S \subset \ZZ_n$ be an inverse closed subset of the additive group $\ZZ_n$ of residue classes modulo $n$, where $0 \in S$. The {\em circulant} graph $\Circ_n(S)$ if the graph with vertex set $\ZZ_n$ in which two vertices $i, j \in \ZZ_n$ are adjacent if and only if $j-i \in S$. 

\section{The alternating jump parameter}
\label{sec:jump}

Let $\G$ be a tetravalent $G$-half-arc-transitive graph for some $G \leq \aut(\G)$. Fix one of the two paired orientations of the edges of $\G$, induced by the action of $G$, and let $r = \rad_{G}(\G)$, $a = \att_{G}(\G)$. For a vertex $v \in V(\G)$ let $C=(u_0,u_1, \ldots, u_{2r-1})$ and $C'=(v_0,v_1, \ldots, v_{2r-1})$ be the two $G$-alternating cycles containing $v$, where $u_0 = v_0 = v$ and $v$ is the tail of the two arcs of $C$, incident to it. By \cite[Lemma~2.6]{Mar98}, the $G$-attachment set $V(C) \cap V(C')$ containing $v$, which we abbreviate by $C \cap C'$ throughout the paper, is
\begin{equation}
\label{eq:intersect}
C \cap C' = \{u_{i\L}: 0\leq i < a\} = \{v_{i\L}:0\leq i<a\},
\end{equation}
where $\L=2r/a$. Define 
$$
q_t(v)=\min\{q:v_{q\L}\in \{u_{\L},u_{-\L}\}\} \mbox{ and } q_h(v)=\min\{q:u_{q\L}\in \{v_{\L},v_{-\L}\}\},
$$ 
where in the case of $a = 1$ this is understood as $q_t(v) = q_h(v) = 0$.
Observe that, by definition, $q_t(v), q_h(v) \leq a/2$. Moreover, since $G$ acts vertex- and edge-transitively on $\G$, the parameters $q_t(v)$ and $q_h(v)$ do not depend on the choice of the vertex $v$. Note also that taking the other of the two $G$-induced orientations of the edges of $\G$ reverses the roles of $q_t(v)$ and $q_h(v)$. For each tetravalent $G$-half-arc-transitive graph $\G$ we can thus define $Q_{G}(\G)=\{q_t,q_h\}$, where $q_t=q_t(v)$ and $q_h=q_h(v)$ for some $v \in V(\G)$ with respect to one of the two $G$-induced orientations of the edges of $\G$. For instance, in the case of loosely $G$-attached graphs, that is when $\att_G(\G)=1$, we have $Q_{G}(\G)=\{0\}$, and in the case of antipodally $G$-attached graphs, that is when $\att_G(\G) = 2$, we have $Q_{G}(\G) = \{1\}$. 

\begin{lemma}
\label{le:mult}
Let $\G$ be a tetravalent $G$-half-arc-transitive graph for some $G \leq \aut(\G)$ and let $r = \rad_{G}(\G)$, $a = \att_{G}(\G)$ and $\L=2r/a$. Fix one of the two $G$-induced orientations of the edges of $\G$ and let $v \in V(\G)$. Let $C=(u_0,u_1, \ldots, u_{2r-1})$ and $C'=(v_0,v_1, \ldots, v_{2r-1})$ be the two $G$-alternating cycles containing $v$, where $u_0 = v_0 = v$, $v$ is the tail of the two arcs of $C$, incident to it, and $u_{\L} = v_{q_t\L}$. Then $u_{i\L}=v_{iq_t\L}$ holds for each $0\leq i < a$. Similarly, depending on whether $v_{\L} = u_{q_h\L}$ or $v_{\L} = u_{-q_h\L}$ holds, we have that $v_{i\L} = u_{iq_h\L}$ holds for each $0\leq i < a$ or $v_{i\L} = u_{-iq_h\L}$ holds for each $0\leq i < a$.
\end{lemma}

\begin{proof}
Observe that if $a \leq 2$, there is nothing to prove. We can thus assume that $a \geq 3$, and so Proposition~\ref{pro:stab} applies. Let $\gamma$ be the unique nontrivial element of $G_{u_\L}$ and observe that, since $G$ is edge-transitive on $\G$, it does not fix any of the neighbors of $u_\L$. Then $\gamma$ fixes both $C$ and $C'$ setwise, and so the restriction of its action to $C$ (respectively $C'$) is the reflection with respect to $u_\L$ (respectively $v_{q_t\L}$). Thus $u_{2\L} = u_0\gamma = v_0\gamma = v_{2q_t\L}$. We can now continue inductively to see that $u_{i\L}=v_{iq_t\L}$ holds for each $0\leq i < a$. The second part of the lemma can be proved analogously.
\end{proof}

The following result shows that the parameters $q_t$ and $q_h$ are in fact nicely related.

\begin{lemma} 
\label{le:inverse}
Let $\G$ be a tetravalent $G$-half-arc-transitive graph for some $G \leq \aut(\G)$, let $a = \att_G(\G)$ and let $Q_{G}(\G)=\{q_t, q_h\}$. Then $\gcd(a,q_t) = \gcd(a,q_h) = 1$ and $q_t q_h \equiv \pm 1 \pmod{a}$.
\end{lemma}

\begin{proof}
We can again assume that $a \geq 3$. Adopt the notation of Lemma~\ref{le:mult}. That $\gcd(a,q_t) = \gcd(a,q_h) = 1$ follows directly from (\ref{eq:intersect}) and Lemma~\ref{le:mult}. To see that $q_t q_h \equiv \pm 1 \pmod{a}$ observe that Lemma~\ref{le:mult} implies that $u_{q_h\L} = v_{q_h q_t\L}$ holds (recall that $q_h < a$). By definition of $q_h$ we thus get $v_{\pm\L} = v_{q_h q_t\L}$, that is $q_h q_t \L \equiv \pm \L \pmod{2r}$.
Since $2r=a\L$,  we obtain $q_hq_t \equiv \pm 1 \pmod{a}$.
\end{proof}

We remark that both $q_hq_t \equiv 1 \pmod{a}$ and $q_hq_t \equiv -1 \pmod{a}$ can occur. For instance, in the smallest half-arc-transitive graph, the well known Doyle-Holt graph $\X_o(3,9;2)$ (see~\cite{AlsMarNow94}), we have $a = 9$ and $Q = \{2,4\}$, and so $q_tq_h \equiv -1 \pmod{a}$, while for both of the graphs $\X_e(4,20;3,0)$ and $\X_e(4,20;3,10)$ we have $a = 20$ and $Q=\{3,7\}$, and so $q_hq_t\equiv 1 \pmod{a}$.

Observe that Lemma~\ref{le:inverse} implies that in fact $q_h, q_t < a/2$ unless $a = 2$ in which case of course $q_h = q_t = 1 = a/2$. Moreover, Lemma~\ref{le:inverse} implies that $q_t$ is uniquely expressible as the smaller of the elements $\pm q_h^{-1}$ in $\ZZ_a$ and that, conversely, $q_h$ is the smaller of the elements $\pm q_t^{-1}$ in $\ZZ_a$. We can thus define the parameter $\jum_G(\G) = \min(Q_G(\G))$ of a tetravalent $G$-half-arc-transitive graph $\G$ where $G \leq \aut(\G)$. We call $\jum_G(\G)$ the {\em $G$-alternating jump} of $\G$. In the case that $G = \aut(\G)$ we abbreviate $\jum_{\aut(\G)}(\G)$ to $\jum(\G)$ and speak of the {\em alternating jump} of $\G$. Note also that Lemma~\ref{le:inverse} implies that the circulants $\Circ_a(\{\pm 1, \pm q_t\})$ and $\Circ_a(\{\pm 1, \pm q_h\})$ are isomorphic, so we say that $\Circ_a(\{\pm 1, \pm \jum_G(\G)\})$ is the {\em associated circulant} of the pair $(\G,G)$, where in the case of $a = 1$ we disregard the loop and consider the associated circulant simply as a one-vertex graph. In the case of $a \leq 2$ the circulant is somewhat degenerate but in general it is tetravalent if and only if $\jum_G(\G) \neq 1$.
This brings us to the following natural question.

\begin{question}
Let $a \geq 3$ be an integer and $1 \leq q < a/2$ be coprime to $a$. Does there exist a tetravalent graph $\G$ with a half-arc-transitive subgroup of automorphisms $G$ such that the associated circulant of the pair $(\G,G)$ is isomorphic to $\Circ_a(\{\pm 1, \pm q\})$?
\end{question}

We remark that the associated circulant of a pair $(\G, G)$, where $\G$ is a tetravalent $G$-half-arc-transitive graph may or may not be arc-transitive. For instance, for the half-arc-transitive graph $\X_o(3,13;3)$ the attachment number and the alternating jump are $13$ and $3$, respectively (see Proposition~\ref{pro:GTA}), and one can easily check that the circulant $\Circ_{13}(\{\pm 1, \pm 3\})$ is not arc-transitive. On the other hand, one can verify that the half-arc-transitive graph $\HAT[500,6]$ from the census~\cite{PotSpiVer15} has attachment number $5$ and the alternating jump parameter $2$, and so the associated circulant $\Circ_5(\{\pm 1, \pm 2\})$, which is the complete graph on $5$ vertices, is arc-transitive. Nevertheless, we think it is worth studying whether the fact that the associated circulant of a pair $(\G, G)$ is or is not arc-transitive has any implications on the graph $\G$. 

It is well known that a tetravalent half-arc-transitive graph of given order is not uniquely determined by its radius and attachment number. For instance, \cite[Theorem~3.4]{Mar98} implies that the graphs $\X_o(6,13;2)$ and $\X_o(6,13;3)$ are both half-arc-transitive with radius $13$ and attachment number $13$. It is not difficult to prove that they are not isomorphic. One of the ways to see this is by inspecting their alternating jump parameter (denoted by $q$ in the remainder of this paragraph). Namely, it can directly be verified (but see also Proposition~\ref{pro:GTA}) that the graph $\X_o(6,13;2)$ has $q = 2$, while the graph $\X_o(6,13;3)$ has $q = 3$. This shows that the parameter $q$ does give a further refinement in the classification of all tetravalent half-arc-transitive graphs. Unfortunately, even the triple $(r,a,q)$ does not uniquely determine a tetravalent half-arc-transitive graph of given order. For instance, \cite[Theorem~1.3, Proposition~9.1]{Spa08} imply that the graphs $\X_e(4,20;3,0)$ and $\X_e(4,20;3,10)$ are both half-arc-transitive with radius $20$ and attachment number $20$, but are not isomorphic. However, both also have $q = 3$, which thus shows that nonisomorphic tetravalent half-arc-transitive graphs with the same order, radius, attachment number and alternating jump parameter exist. Nevertheless, we show in the remainder of this paper that the alternating jump parameter does give a very useful insight into the structure of tetravalent graphs admitting a half-arc-transitive group of automorphisms.

 \section{The alternating jump of tightly attached graphs}
 \label{sec:GTA}
 
The examples from the last paragraph of the previous section are all tightly attached half-arc-transitive graphs. The fact that their alternating jump parameter is actually one of their defining parameters is not a coincidence. In this section we show that the alternating jump parameter of tetravalent graphs admitting a half-arc-transitive group action relative to which they are tightly attached is determined by their defining parameters from the classifications given in~\cite{Mar98, Spa08}. Depending on whether the corresponding radius is odd or even, respectively, the classification of such graphs is given in the following two propositions, which can be extracted from \cite[Propositions 3.2, 3.3]{Mar98} and \cite[Theorem~1.2]{Spa08}, respectively (see also~\cite[Theorem~4.5]{MarPra99} and~\cite[Theorem~8.1]{Wil04}). 

\begin{proposition}{\cite{Mar98}} \label{pr:GTAodd}
A tetravalent graph $\G$ admits a half-arc-transitive group of automorphisms relative to which it is tightly attached of odd radius $r$ if and only if $\G \cong \X_o(m,r;q)$ for some integer $m \geq 3$ and some $q \in \ZZ^*_r$ with $q^m \equiv \pm 1 \pmod{r}$.
\end{proposition}

\begin{proposition}{\cite{Spa08}} 
\label{pr:GTAeven}
A tetravalent graph $\G$ admits a half-arc-transitive group of automorphisms relative to which it is tightly attached of even radius $r$ if and only if either $r = 2$ and $\G$ is isomorphic to a lexicographic product of a cycle with $2K_1$, or $r \geq 4$ and $\G \cong \X_e(m, r; q, t )$, where $m \geq 4$ is even and $q \in \ZZ^*_r$, $t \in \ZZ _r$ are such that $q^m =1$, $t(q-1)=0$ and $1 + q + \cdots + q^{m-1} + 2t = 0$.
\end{proposition}

Before we determine the alternating jump parameter of the graphs $\G$ from the above two propositions we first fix some notation. We first point out that the action of the corresponding half-arc-transitive group $G$ of automorphisms, both in the graphs $\X_o(m,r;q)$ and the graphs $\X_e(m,r;q,t)$, is such that the corresponding $G$-attachment sets are $\G_i = \{u_i^j \in V(\G) \colon j \in \ZZ_r\}$, where $i \in \ZZ_m$ (see~\cite{Mar98, Spa08} for details). In particular, in one of the two $G$-induced orientations of the edges of $\G$ all edges are oriented from $\G_i$ to $\G_{i+1}$ for all $i \in \ZZ_m$. In the proof of the following result we always choose this orientation. The notation $\min_r\{\pm q, \pm q^{-1}\}$ in the statement of the following proposition stands for the minimal integer from the set $\{0,1,\ldots , r-1\}$, which is congruent to one of $q, -q, q^{-1}, -q^{-1}$ modulo $r$, where $q^{-1}$ is the inverse of $q$ modulo $r$.

\begin{proposition}
\label{pro:GTA}
Let $\G$ be a tetravalent graph admitting a half-arc-transitive group $G$ of automorphisms relative to which it is tightly attached with radius $r \geq 3$. Then $\jum_G(\G) = \min_r\{\pm q, \pm q^{-1}\}$, where $\G \cong \X_o(m, r; q)$ or $\G \cong \X_e(m, r; q, t )$, depending on whether $r$ is odd or even, respectively.
\end{proposition}

\begin{proof} 
Recall that, since the $G$-attachment sets are blocks of imprimitivity for the action of $G$ on $\G$, the radius $r$ divides $|V(\G)|$. Let $m$ be such that $|V(G)|=mr$ and note that, by \cite[Proposition~2.4]{Mar98}, we have $m\geq 3$. We separate the proof depending on the parity of $r$. 

Suppose first that $r$ is odd. In this case Proposition~\ref{pr:GTAodd} implies that $\G \cong \X_o(m,r;q)$ for some $q \in \ZZ^*_r$ such that $q^m = \pm 1$. Choose the orientation of $\G$ described in the paragraph preceding Proposition~\ref{pro:GTA} and let $v = u_1^0$. In order to determine $\jum_G(\G)$ we find the values $q_t(v)$ and $q_h(v)$ (recall that the parameters $q_t$ and $q_h$ do not depend on the choice of the vertex $v$). By definition of the graph $\X_o(m,r;q)$ the two $G$-alternating cycles containing $v$ are 
$$
C = (u_{1}^{0},u_{2}^{q}, u_{1}^{2q},u_{2}^{3q},\ldots, u_1^{r-2q}, u_{2}^{r-q})
	\mbox{ and }
C' = (u_{1}^{0},u_{0}^{1}, u_{1}^{2},u_{0}^{3},\ldots,u_{1}^{r-2}, u_{0}^{r-1}).
$$ 
Since $r$ is odd and $q$ is coprime to $r$, the corresponding $G$-attachment set $C \cap C'$ is then 
$$
C \cap C' = \{u_1^{2 j q} \colon j \in \ZZ_r\} \cap \{u_{1}^{2 j} \colon j \in \ZZ_r\} = \{u_1^j \colon j \in \ZZ_r\}.
$$ 
It follows that $q_t(v)=\min\{s \colon u_{1}^{2s}\in \{u_1^{2q}, u_1^{-2q}\}\}$, which implies $q_t(v) = \min \{q,-q\}$. Therefore, Lemma~\ref{le:inverse} implies $q_h(v)\in\{q^{-1},-q^{-1}\}$, and so $\jum_G(\G) = \min_r\{\pm q, \pm q^{-1}\}$, as claimed.

Suppose now that $r$ is even. Since $r \geq 3$ Proposition~\ref{pr:GTAeven} implies that $\G \cong \X_e(m, r; q, t )$ for some $q \in \ZZ^*_r$ and $t \in \ZZ _r$ such that $q^m =1$, $t(q-1)=0$ and $1 + q + \cdots + q^{ m-1} + 2t = 0$.  Then, choosing the orientation of $\G$ described in the paragraph preceding Proposition~\ref{pro:GTA}, the two $G$-alternating cycles containing the vertex $v = u_{1}^{0}$ are 
$$
C = (u_{1}^{0},u_{2}^{q}, u_{1}^{q},u_{2}^{2q},u_1^{2q}\ldots, u_1^{r-q}, u_{2}^{0})
	\mbox{ and }
C'=(u_{1}^{0},u_{0}^{0}, u_{1}^{1},u_{0}^{1},\ldots,u_{1}^{r-1}, u_{0}^{r-1}).
$$
A similar argument as in the case of $r$ being odd proves that $\jum_G(\G) = \min_r\{\pm q, \pm q^{-1}\}$.
\end{proof}

By \cite[Proposition 4.1]{Mar98} for any $m \geq 3$, any odd $r \geq 3$ and any $q \in \ZZ_r$ such that $q^m = \pm 1$ the graph $\X_o(m,r;q)$ is isomorphic to each of the graphs $\X_o(m,r;-q)$, $\X_o(m,r;q^{-1})$ and $\X_o(m,r;-q^{-1})$. Similarly, \cite[Proposition~3.9]{Spa08} shows that for any pair of even integers $m,r \geq 4$ and any $q, t \in \ZZ_r$ with $q^m = 1$, $t(q-1) = 0$ and $1+q+\cdots + q^{m-1} + 2t=0$ the graph $\X_e(m,r;q,t)$ is isomorphic to each of the graphs $\X_e(m,r;q^{-1},t)$, $\X_e(m,r;-q, t + q + q^3 + \cdots + q^{m-1})$ and $\X_e(m,r;-q^{-1}, t + q + q^3 + \cdots + q^{m-1})$. We thus have the following corollary.

\begin{corollary}
\label{cor:GTA}
Let $\G$ be a tetravalent graph admitting a half-arc-transitive group $G$ of automorphisms relative to which it is tightly attached with radius $r \geq 3$. Let $m$ be such an integer that the order of $\G$ is $mr$ and let $q = \jum_G(\G)$. Then, one of (i) and (ii) below holds, depending on whether $r$ is odd or even, respectively:
\begin{itemize}
\item[(i)] $\G \cong \X_o(m,r;q)$.
\item[(ii)] $\G \cong \X_e(m,r;q,t)$, where $t \in \ZZ_r$ is one of the two solutions of the equation $1+q+\cdots + q^{m-1} + 2t = 0$,
\end{itemize}
\end{corollary}

The above corollary shows that, at least for tetravalent graphs admitting a half-arc-transitive group of automorphisms relative to which they are tightly attached, the order, together with the corresponding radius and alternating jump parameter (almost) completely determine the graph. More precisely, they determine it completely in the case of odd radius while in the case of even radius at most two such graphs can exist.

We want to point out however, that a given tetravalent graph $\G$ may admit more than one half-arc-transitive group of automorphisms (even such that it is tightly attached with respect to more than one of them) and that the corresponding radii (nor alternating jump parameters) need not be the same. For instance in~\cite[Example 2.1]{MarPra99} (see also \cite{MarNed02}) it was shown that the well-known wreath graphs $C_{2r}[2K_1]$, where $r\geq 3$, admit two different half-arc-transitive groups of automorphisms relative to which they are tightly attached. For one of them the corresponding radius is $2$ while for the other one it is $r$. There are other such examples that can be found in~\cite{MarPra99} but one can also find others by going through the census from~\cite{PotSpiVer15}. For instance, the graph GHAT$[294;1]$ from the census admits three different half-arc-transitive groups of automorphisms. It is loosely attached with radii $3$ and $7$ for two of them and is antipodally attached with radius $3$ for the third one. Of course, if we restrict to half-arc-transitive graphs (where the full automorphism group is half-arc-transitive), these phenomena cannot occur. 

\section{The graph of alternating cycles and the quotient graph}
\label{sec:kernel}

In this section we use the alternating jump parameter to determine the kernel of the action of $G$ on the graph of alternating cycles $\Alt_G(\G)$ from~\cite{PotSpa17} and the quotient graph $\G_{\B}$ from~\cite{MarPra99}, where $\G$ is a tetravalent graph admitting a half-arc-transitive group of automorphisms $G$ (see Theorem~\ref{the:kernel} and Theorem~\ref{the:allkernels}). 

We first recall the definition of the graph of alternating cycles $\Alt_G(\G)$ from~\cite{PotSpa17}. Let $\G$ be a tetravalent graph admitting a half-arc-transitive group of automorphisms $G$. The vertex set of $\Alt_G(\G)$ is then the set of all $G$-alternating cycles of $\G$ with two such cycles being adjacent whenever they have a non-empty intersection. Of course, $G$ has a natural action on the graph $\Alt_G(\G)$. In fact, it was proved in \cite[Proposition 4]{PotSpa17} that the induced action is vertex- and edge-transitive and is half-arc-transitive if and only if $\att_G(\G)$ divides $\rad_G(\G)$. One of the aims of this section is to determine the kernel $K_G(\Alt_G(\G))$ of this action. Observe that $K_G(\Alt_G(\G))$ is the normal subgroup of $G$ consisting of all elements fixing each $G$-alternating cycle (throughout the rest of paper we say that an automorphism of $\G$ {\em fixes} a cycle if it maps the cycle as a subgraph to itself). The following lemma gives a sufficient condition for an element of $G$ to be contained in $K_G(\Alt_G(\G))$.

\begin{lemma} 
\label{le:rotation}
Let $\G$ be a tetravalent $G$-half-arc-transitive graph for some $G \leq \aut(\G)$ such that $3 \leq a < r$, where $r = \rad_G(\G)$ and $a = \att_G(\G)$. Suppose $\gamma \in G$ fixes some $G$-alternating cycle $C$, as well as all the $G$-attachment sets containing vertices of $C$. Then $\gamma \in K_G(\Alt_G(\G))$.
\end{lemma}

\begin{proof}
Denote $C=(u_0,u_1, \ldots, u_{2r-1})$, fix one of the two $G$-induced orientations of the edges of $\G$ and let $Q_G(\G) = \{q_t,q_h\}$. Since $\gamma$ fixes $C$, as well as all the $G$-attachment sets containing vertices of $C$, there exists a unique $0 \leq k < a$ such that $u_0\gamma = u_{k\L}$, where $\L = 2r/a$ (recall that the $G$-attachment sets are of the form given in (\ref{eq:intersect})). In fact, by exchanging the roles of $u_j$ and $u_{-j}$ for each $j$ if necessary, we can assume $k \leq a/2$. Since $u_1$ is a vertex of the induced cycle $C$ (see \cite[Proposition 2.4]{Mar98}) and it is adjacent to $u_0$, the image $u_1\gamma$ is one of the vertices $u_{k\L + 1}$ and $u_{k\L - 1}$. Moreover, since $a < r$, we have $\L \geq 3$, and so the $G$-attachment set containing $u_1$, namely $\{u_{1+i\L} \colon 0\leq i<a\}$, does not contain $u_{k\L - 1}$. It follows that $u_1\gamma = u_{k\L + 1}$, and so the restriction of the action of $\gamma$ to $C$ is the $k\L$-step rotation such that $u_j\gamma = u_{j + k\L}$ for all $0 \leq j \leq 2r-1$.

Let $C' = (v_0,v_1, \ldots, v_{2r-1})$ be one of the $G$-alternating cycles having a non-empty intersection with $C$. Without loss of generality assume $v_0 = u_0$ and $u_\L = v_{q\L}$, where $q$ is one of $q_t$ and $q_h$, depending on whether $u_0$ is the tail of the two arcs of $C$ incident to it or not, respectively. By Lemma~\ref{le:mult} we have $u_{i\L} = v_{iq\L}$ for all $0 \leq i < a$, and so $v_0\gamma = u_0\gamma = u_{k\L} = v_{kq\L}$. Recall that $v_\L = u_{q^{-1}\L}$, where $q^{-1}$ is the inverse of $q$ modulo $a$. Thus 
$$
	v_\L\gamma = u_{q^{-1}\L}\gamma = u_{q^{-1}\L + k\L} = u_{(q^{-1} + k)\L} = v_{(q^{-1} + k)q\L} = v_{\L + kq\L}.
$$  
Since $\L < r$ and $\gamma$ fixes $C'$ (since it fixes $C$ and the attachment set $C \cap C'$), it is now clear that the restriction of the action of $\gamma$ to $C'$ is the $kq\L$-step rotation such that $v_j\gamma = v_{j+kq\L}$ for all $0 \leq j \leq 2r-1$. It follows that $\gamma$ fixes all of the $G$-attachment sets containing vertices of $C'$, and so the assumptions of the lemma hold for $C'$ and $\gamma$ as well. By connectedness $\gamma$ fixes each $G$-alternating cycle of $\G$, and so $\gamma \in K_G(\Alt_G(\G))$.
\end{proof}

We can now (almost) completely determine the kernel $K_G(\Alt_G(\G))$ for a tetravalent graph $\G$ admitting a half-arc-transitive group of automorphisms $G \leq \aut(\G)$. 

\begin{theorem}
\label{the:kernel}
Let $\G$ be a tetravalent $G$-half-arc-transitive graph for some $G \leq \aut(\G)$ and let $r = \rad_G(\G)$ and $a = \att_G(\G)$. Let $K = K_G(\Alt_G(\G))$ be the kernel of the action of $G$ on the graph $\Alt_{G}(\G)$ of $G$-alternating cycles of $\G$. Then one of the following holds:
\begin{enumerate}\itemsep = 0pt
\item[(i)] $a = 2r$ and $K = D_r$;
\item[(ii)] $a = r = 2$, in which case $\G$ is the wreath graph $C_n[2K_1]$ for some integer $n$, and $K$ is isomorphic to a subgroup of the elementary abelian $2$-group of order $2^n$; 
\item[(iii)] $a = r > 2$ and $K$ is isomorphic to the dihedral group of order $2a$; 
\item[(iv)] $a < r$ with $a\mid r$ and $K$ is isomorphic to the cyclic group of order $a$, unless possibly if $a = 2$, in which case the kernel $K$ can be trivial; 
\item[(v)] $a < r$ with $a\nmid r$ and $K$ is isomorphic to the cyclic group of order $a/2$.
\end{enumerate}
\end{theorem}

\begin{proof}
By \cite[Proposition~2.4]{Mar98} the graph $\Alt_G(\G)$ has at least three vertices unless $a = 2r$, in which case there are exactly two $G$-alternating cycles (each containing all of the vertices of $\G$) and $K = D_r$ is the dihedral group of order $2r$ (while $G$ itself is an extension of $K$ by $\ZZ_2$). Moreover, the case $a = r$, that is when $\G$ is a tightly $G$-attached graph, is settled by \cite[Proposition~3.1]{Mar98} since in this case the kernel $K$ clearly coincides with the kernel of the action of $G$ on the set of all $G$-attachment sets. For the rest of the proof we can thus assume that $\G$ has at least three $G$-alternating cycles and that $1 \leq a < r$. Observe that in this case the action of $K$ on $\G$ is semiregular since $a < r$ implies that no two neighbors of a vertex of $\G$ can belong to the same pair of $G$-alternating cycles.

Next, observe that if $a = 1$, that is if $\G$ is loosely $G$-attached, the kernel $K$ is trivial. We can thus further assume that $a \geq 2$. For the rest of the proof we adopt the notation from Lemma~\ref{le:mult}. We set $\L=2r/a$, fix one of the two $G$-induced orientations of the edges of $\G$, set $Q_G(\G) = \{q_t, q_h\}$ and we let $v \in V(\G)$. Furthermore, we let $C = (u_0,u_1, \ldots, u_{2r-1})$ and $C'=(v_0,v_1, \ldots, v_{2r-1})$ be the two $G$-alternating cycles containing $v = u_0 = v_0$ where $v$ is the tail of the two arcs of $C$ incident to it and $u_{\L} = v_{q_t\L}$. 

Suppose first that $a = 2$. If $r$ is odd (in which case $a \nmid r$), then $u_0$ is the tail, while $u_r$ is the head of the two arcs of $C$ incident to $u_0$ and $u_r$, respectively. It follows that no element of $K$ can map $u_0$ to $u_r$, and so $K$ is trivial in this case. If however $r$ is even, then the fact that the intersection of any two adjacent $G$-alternating cycles is a pair of antipodal vertices on both of them implies that the automorphism $\tau$, mapping each vertex to its antipodal counterpart on both $G$-alternating cycles containing it (see \cite[Proposition~7]{PotSpa17}), is the only possible nontrivial element of $K$. Depending on whether $\tau$ is or is not contained in $G$, the kernel $K$ is either the cyclic group of order $2$ or is trivial, respectively. 

We are left with the case $3 \leq a < r$. Recall that the vertex $u_\L$ is the tail of the two arcs of $C$ incident to it if and only if $\ell$ is even, which occurs if and only if $a$ divides $r$. Since $a \geq 3$, Proposition~\ref{pro:stab} implies that $G$ acts regularly on the edge set of $\G$. If $\L$ is even let $\rho \in G$ be the unique element mapping the arc $(u_0, u_1)$ to the arc $(u_\L, u_{\L+1})$ and if $\L$ is odd let $\rho \in G$ be the unique element mapping the arc $(u_0, u_1)$ to $(u_{2\L}, u_{2\L+1})$. It is clear that the restriction of the action of $\rho$ to $C$ is an $\L$-step or $2\L$-step rotation, depending on whether $a$ divides $r$ or not, respectively. In both cases $\rho$ fixes $C$ as well as all of the $\L$ attachment sets containing the vertices of $C$. Thus $\rho$ satisfies the conditions of Lemma~\ref{le:rotation}, and so $\rho \in K$. Since $C \cap C' = \{u_{i\L} \colon 0 \leq i < a\}$ and the action of $K$ on $\G$ is semiregular, it follows that $K = \langle \rho \rangle$, which completes the proof.   
\end{proof}

We point out that for $a = 2$ and $r$ even both possibilities from item (iv) of the above theorem can indeed occur. Going through the census of all tetravalent graphs of order up to $1000$, admitting a half-arc-transitive subgroup of automorphisms~\cite{PotSpiVer15}, one can verify that in most cases where $a = 2$ and $r$ is even the kernel $K_G(\Alt_G(\G))$ is nontrivial (and is thus of order $2$). However, there are examples where it is in fact trivial. The smallest such example in the above mentioned census is the graph $\G =$ $\GHAT[162,1]$ of order $162$ whose automorphism group $\aut(\G)$ is of order $1296$. The group $\aut(\G)$ acts arc-transitively on $\G$ but it has two (normal) subgroups $G_1$ and $G_2$ of orders $324$ and $648$, respectively, acting half-arc-transitively on $\G$, for both of which the corresponding radius is $6$ and the attachment number is $2$. However, for the group $G_1$ the corresponding kernel $K_G(\Alt_G(\G))$ is trivial, while for the group $G_2$, it is of order $2$. Observe that, since the automorphism $\tau$ from \cite[Proposition~7]{PotSpa17} centralizes $G_1$, it is clear that $G_2 = G_1 \times \langle \tau \rangle$. This of course will always be the case, that is, even if the situation from Theorem~\ref{the:kernel}~(iv) with $a = 2$, $r$ even and $K$ trivial occurs for some half-arc-transitive subgroup $G \leq \aut(\G)$, the group $G \times \langle \tau \rangle$ will also act half-arc-transitively on $\G$ with the same natural orientation of the edges and with the corresponding kernel being $\langle \tau \rangle$ and thus of order $2$. Nevertheless, the following problem naturally arises.

\begin{problem}
Classify all tetravalent graphs $\G$, admitting a half-arc-transitive subgroup $G \leq \aut(\G)$ with $\att_G(\G) = 2$ and $\rad_G(\G)$ even, such that $K_G(\Alt_G(\G))$ is trivial. In particular, does a half-arc-transitive graph $\G$ with these properties exist?
\end{problem}

We now focus on the quotient graph $\G_{\B}$ from~\cite{MarPra99} (the graph $\G_\B$ was denoted by $\G_\Sigma$ in~\cite{MarPra99}) and the kernel of the action of $G$ on it. We first recall the definition. For ease of reference in the remainder of the paper we formulate this as a construction.

\begin{construction}[\cite{MarPra99}]
\label{cons:quot}
Let $\G$ be a tetravalent graph admitting a half-arc-transitive group of automorphisms $G \leq \aut(\G)$ and let $r = \rad_{G}(\G)$, $a = \att_{G}(\G)$ and $\L=2r/a$. Define 
$$
s=\left\{ \begin{array}{lcl}
\L/2 & : & \mbox{  if } \L \mbox{  is even};\\
\L &  : &  \mbox{  if } \L \mbox{  is odd}.
\end{array}
 \right.
$$ 
Let $C=(u_0,u_1,\ldots,u_{2r-1})$ be a $G$-alternating cycle. For each $u_i \in V(C)$ set $B_s(C;u_i)=\{u_{i+2sj} \colon 0\leq j < a\}$ (where the indices are computed modulo $2r$). Let $u = u_i \in V(C)$ for some $i \in \{0,1,\ldots,2r-1\}$ and note that if $\L$ is even the set $B_s(C;u)$ is precisely the $G$-attachment set containing $u$, while if $\L$ is odd the $G$-attachment set containing $u$ is $B_s(C;u) \cup B_s(C,u_{i+\L})$. Moreover, if $\L$ is odd and $u$ is the tail of the two arcs of $C$ incident to it, then all of the vertices in $B_s(C;u)$ also have this property, while all of the vertices in $B_s(C;u_{i+\L})$ are heads of the two arcs of $C$ incident to them. The set $B_s(C;u)$ thus forms a block of imprimitivity for $G$, and so $\B = \{B_s(C;u)\gamma \colon\gamma \in G\}$ is an imprimitivity block system for $G$ (note that it does not depend on the choice of the $G$-alternating cycle $C$ nor the vertex $u$ of $C$). The graph $\G_{\B}$ is then defined as the quotient graph of $\G$ with respect to $\B$, whose vertex set coincides with $\B$ with two blocks $B$ and $B'$ from $\B$ being adjacent whenever there is a pair of adjacent vertices $u \in B$ and $v \in B'$. 
\end{construction}
 
Before stating and proving the next theorem in which we compare the kernel of the action of $G$ on $\Alt_G(\G)$ and $\G_\B$ we fix the following notation. Let $\cal{A}$ be the set of all $G$-attachment sets of $\G$ (note that $\mathcal{A} = \B$ if and only if $a$ divides $r$). Then $G$ has a natural action on the graph $\G_{\B}$ as well as on the set $\cal{A}$. Let $K_G(\G_{\B})$ and $K_G(\cal{A})$ be the kernels of these actions. The following result shows that these two kernels are closely related to the kernel $K_G(\Alt_G(\G))$. 

\begin{theorem}
\label{the:allkernels}
Let $\G$ be a tetravalent $G$-half-arc-transitive graph for some $G\leq \aut(\G)$ and let $r = \rad_G(\G)$ and $a = \att_G(\G)$. Let $\G_\B$ be as in Construction~\ref{cons:quot} and let $\cal{A}$ be the partition of $V(\G)$ described in the preceding paragraph. If $a<2r, $ then
$$
K_G(\G_{\B})=K_G(\Alt_G(\G))=K_G(\cal{A}).
$$
\end{theorem}

\begin{proof}
Let $\alpha \in K_G(Alt_G(\G))$. Then $\alpha$ fixes each $G$-alternating cycle, and so it also fixes all the $G$-attachment sets since they are the intersections of $G$-alternating cycles. Thus $K_G(\Alt_G(\G))\subseteq K_G(\mathcal{A})$. 

Suppose now there exists $\gamma \in K_G(\mathcal{A}) \backslash K_G(\Alt_G(\G))$. Let $\L=2r/a$, let $C=(u_0,u_1, \ldots, u_{2r-1})$ be a $G$-alternating cycle such that $C\gamma \neq C$ and let $C'=(v_0,v_1, \ldots, v_{2r-1})$ be the other $G$-alternating cycle containing the vertex $u_0$, where $u_0=v_0$. Since $\gamma \in K_G(\mathcal{A})$, it fixes the $G$-attachment set $C \cap C'$ setwise, and so it interchanges the cycles $C$ and $C'$. Note that $a<2r$ implies that the vertex $u_1$ is not contained on $C'$ but on some $G$-alternating cycle $C''$ with $C''\neq C$ and $C''\neq C'$. But since $\gamma$ also fixes setwise the $G$-attachment set $C\cap C''$ we get $C\cap C''=C\gamma\cap C'' \gamma=C' \cap C''\gamma$, which is impossible. Therefore, $K_G(\Alt_G(\G))=K_G(\mathcal{A})$, as claimed.

To complete the proof we need to verify that $K_G(\G_\B)=K_G(\mathcal{A})$. Of course, if $a$ divides $r$ then $\L$ is even in which case $\B$ coincides with $\cal{A}$, and so there is nothing to prove. We can thus assume that $a \nmid r$. 
In this case each $G$-attachment set is the disjoint union of two elements of $\B$, and so it is clear that $K_G(\G_{\B}) \leq K_G(\mathcal{A})$. By the first part of the proof it thus suffices to prove that $K_G(\Alt_G(\G))  \leq K_G(\G_{\B})$.  Let $\gamma \in K_G(\Alt_G(\G))$ and let $C = (u_0,u_1, \ldots, u_{2r-1})$ be a $G$-alternating cycle. Let $u = u_i$ for some $i \in \{0,1,\ldots,2r-1\}$ and let $C'$ be the other $G$-alternating cycle containing $u$. Since in this case $\L$ is odd, $C \cap C' = B_s(C;u_i) \cup B_s(C,u_{i+\L})$. Since $\gamma$ fixes both $C$ and $C'$, each $u_{i+j\L}$ is mapped by $\gamma$ to a vertex $u_{i+j'\L}$ such that $j'$ has the same parity as $j$ (otherwise $\gamma$ interchanges $C$ and $C'$). Therefore $\gamma$ fixes each of $B_s(C;u_i)$ and $B_s(C,u_{i+\L})$ setwise, and so it fixes all the elements of $\B$ setwise. It thus follows that $K_G(\Alt_G(\G))  \leq K_G(\G_{\B})$, which completes the proof.
\end{proof}

As was pointed out in the Introduction, the quotient graph $\G_\B$ from Construction~\ref{cons:quot} is of great importance in the study of tetravalent graphs admitting a half-arc-transitive group of automorphisms. Namely, by~\cite[Theorem~3.6]{MarPra99} a tetravalent graph $\G$ admitting a half-arc-transitive group of automorphisms $G \leq \aut(\G)$ with $\att_G(\G) \neq 2\rad_G(\G)$ is either tightly $G$-attached or the quotient graph $\G_\B$ is a tetravalent graph admitting a half-arc-transitive action of a quotient group of $G$, relative to which it is loosely or antipodally attached. Of course, this quotient group is precisely the group $G/K_G(\G_\B)$. Combining together Theorems~\ref{the:kernel} and~\ref{the:allkernels} we can make an important improvement of \cite[Theorem~3.6]{MarPra99}. Namely, we now know that the imprimitivity block system $\B$ actually coincides with the set of orbits of the cyclic normal subgroup $K_G(\G_\B)$ (recall that in the case that $\rad_G(\G)$ is even, $\att_G(\G) = 2$ and $K_G(\Alt_G(\G))$ is trivial, we first need to enlarge the group $G$ to $G \times \langle \tau\rangle$). We thus obtain the following result.

\begin{theorem}
\label{the:quotient}
Let $\G$ be a tetravalent $G$-half-arc-transitive graph for some $G \leq \aut(\G)$ such that $a \neq 2r$, where $r = \rad_G(\G)$ and $a = \att_G(\G)$. In the case that $r$ is even, $a = 2$ and the automorphism $\tau$ of $\G$, interchanging all pairs of antipodal vertices on the $G$-alternating cycles of $\G$, is not contained in $G$, replace $G$ by $G \times \langle \tau \rangle$. Let $\B$ and $\G_\B$ be as in Construction~\ref{cons:quot}. Then precisely one the following holds:
\begin{itemize}
\itemsep = 0pt
\item[{\rm (i)}] $a = r$, that is $\G$ is tightly $G$-attached;
\item[{\rm (ii)}] $a < r$ and the partition $\B$ coincides with the orbits of the cyclic normal subgroup $K_G(\G_\B)$ which is of order $a$ or $a/2$, depending on whether $a$ divides $r$ or not, respectively. Moreover, letting $\bar{G} = G/K_G(\G_\B)$, the quotient graph $\G_\B$ is a tetravalent $\bar{G}$-half-arc-transitive graph which is loosely $\bar{G}$-attached or antipodally $\bar{G}$-attached, depending on whether $a$ divides $r$ or not, respectively.
\end{itemize}
\end{theorem}

Theorem~\ref{the:quotient} provides the link between the frameworks for a possible classification of tetravalent graphs admitting a half-arc-transitive group of automorphism, started in~\cite{MarPra99,AlAlMuPrSp2016}. Namely, for a tetravalent graph $\G$ admitting a half-arc-transitive group of automorphism $G$ such that $\G$ does not have one of the three special attachment numbers (1,2 or $\rad_G(\G)$), Theorem~\ref{the:quotient} ensures the existence of a non trivial cyclic normal subgroup of $G$ such that the quotient graph with respect to its orbits is again a tetravalent graph admitting a half-arc-transitive group of automorphism with attachment number 1 or 2. Therefore, in order to describe or classify all the basic graphs in the sense of~\cite{AlAlMuPrSp2016} the tetravalent graphs admitting a half-arc-transitive group of automorphism relative to which they are loosely or antipodally attached will have to be thoroughly investigated. 
\medskip

In the remainder of this section we give two additional results on the graphs $\Alt_G(\G)$ and $\G_\B$, corresponding to a tetravalent $G$-half-arc-transitive graph $\G$ where $G \leq \aut(\G)$. Moreover, by making a careful analysis of the graphs from the census~\cite{PotSpiVer15} using a suitable package such as {\sc Magma}~\cite{Mag} several interesting questions about these two graphs arise which provide a possible direction for future research. We formulate some of these questions and give a few examples of graphs with interesting properties. 

Let $\G$ be a tetravalent $G$-half-arc-transitive graph where $G \leq \aut(\G)$ and let $a = \att_G(\G)$ and $r = \rad_G(\G)$. As already pointed out the group $G$ induces a natural action on both $\Alt_G(\G)$ and $\G_\B$. The quotient group $G/K_G(\Alt_G(\G))$ acts vertex- and edge-transitively on $\Alt_G(\G)$ and this action is arc-transitive if and only if $a$ does not divide $r$ and is half-arc-transitive otherwise (see~\cite[Proposition~4]{PotSpa17}). Of course, even if $a$ does divide $r$ the graph $\Alt_G(\G)$ may be arc-transitive. A similar situation holds for $\G_\B$, but in this case $G/K_G(\G_\B)$ always acts half-arc-transitively. As we now prove, a natural way for $\Alt_G(\G)$ and $\G_\B$ to be arc-transitive is that there is an automorphism of $\G$, interchanging the two $G$-induced orientations of the edges of $\G$.

\begin{proposition}
\label{le:int} 
Let $\G$ be a tetravalent $G$-half-arc-transitive graph for some $G \leq \aut(\G)$. If $\G$ is tightly $G$-attached or $\aut(\G)$ contains an element interchanging the two $G$-orbits on $A(\G)$, then $\Alt_G(\G)$ and $\G_{\B}$ are arc-transitive graphs. 
\end{proposition}

\begin{proof} 
Note that if $\G$ is tightly $G$-attached, then $\Alt_G(\G)$ and $\G_{\B}$, are both isomorphic to the cycle of length $|V(\G)|/\rad_G(\G)$.  For the rest of the proof we can thus assume that $\G$ is not tightly $G$-attached.

Let $\O_1$ and $\O_2$ be the two $G$-orbits on $A(\G)$. Suppose that $\aut(\G)$ contains an element $\alpha$ interchanging $\O_1$ and $\O_2$ and let $(u,v) \in \O_1$. Since $\Alt_G(\G)$ and $\G_{\B}$ admit a half-arc-transitive action, it suffices to find an element of their automorphism groups which interchanges two adjacent vertices. Moreover, in terms of vertices and edges the $G$-alternating cycles with respect to $\O_1$ coincide with the $G$-alternating cycles with respect to $\O_2$, and so $\alpha$ induces an automorphism of both $\Alt_G(\G)$ and $\G_\B$. Let $C_1=(u_0,u_1,\ldots,u_{2r-1})$ and $C_2=(v_0,v_1,\ldots,v_{2r-1})$ be the two $G$-alternating cycles such that $u_0=v_0=u$ and $u_1=v$. Note that there exists $\beta \in G$ such that $(u_0,u_1)\alpha\beta = (u_0,v_1)$. Then $\alpha\beta$ maps $C_1$ to $C_2$, and since it fixes $u_0$, it maps $C_2$ to $C_1$, proving that $\Alt_G(\G)$ is arc-transitive. Similarly, there exists $ \gamma \in G$ such that $(u_0,u_1)\alpha\gamma = (u_1,u_0)$. Then $\alpha\gamma$ clearly interchanges the adjacent vertices of $\G_\B$  containing the vertices $u_0$ and $u_1$, respectively. Therefore the graph $\G_\B$ is also arc-transitive.
\end{proof}
  
Let us mention, that the situation from Proposition~\ref{le:int} seems to be quite common. Namely, by going through the census of all arc-transitive tetravalent graphs $\G$ admitting a half-arc-transitive group $G$ of automorphisms up to order $1000$, one finds that most of them admit automorphisms interchanging the two $G$-orbits on $A(\G)$. Nevertheless, there are examples where this does not hold. For instance, the arc-transitive graph $\G = \GHAT[21,1]$ is tightly attached with respect to a suitable half-arc-transitive subgroup of automorphisms, but there is no element of $\aut(\G)$ interchanging the corresponding orbits on $A(\G)$. The same occurs for the graph $\GHAT[252,14]$, which turns out to be the smallest example with this property that is not tightly attached with respect to the half-arc-transitive subgroup of automorphisms (it is loosely attached). 

Of course, each of the graphs $\Alt_G(\G)$ and $\G_\B$ can be arc-transitive even if no automorphism of $\G$ interchanges the two $G$-orbits on $A(\G)$. This turns out to be the case also for the above mentioned graphs $\GHAT[21,1]$ and $\GHAT[252,14]$. In fact, by going through the census from~\cite{PotSpiVer15} one finds that for all arc-transitive tetravalent graphs admitting a half-arc-transitive group $G \leq \aut(\G)$ up to order $1000$, the graphs $\Alt_G(\G)$ and $\G_\B$ are both arc-transitive. This suggests the following natural question.

\begin{question}
\label{que:AT}
Is it true that if $\G$ is an arc-transitive tetravalent graph admitting a half-arc-transitive group $G \leq \aut(\G)$, the graphs $\Alt_G(\G)$ and $\G_{\B}$ are both arc-transitive?
\end{question}

It is also interesting to investigate what can be said about arc-transitivity or half-arc-transitivity of the graphs $\Alt_G(\G)$ and $\G_\B$ in the case that $\G$ is half-arc-transitive. Going through the census of all tetravalent half-arc-transitive graphs of order up to $1000$ one finds that there are examples $\G$ for which $\Alt(\G)$ and/or $\G_{\B}$ are half-arc-transitive. For instance, for $\G = \mathrm{HAT}[500;6]$ (which has $\rad(\G) = 25$ and $\att(\G) = 5$) the graph $\G_{\B}$ is half-arc-transitive while $\Alt(\G)$ is arc-transitive. For $\G = \mathrm{HAT}[600;4]$ (which has $\rad(\G) = 6$ and $\att(\G) = 2$) the graphs $\G_{\B}$ and $\Alt(\G)$ are both half-arc-transitive. Finally, for $\G = \mathrm{HAT}[84;1]$ (which has $\rad(\G) = 14$ and $\att(\G) = 7$) both $\G_{\B}$ and $\Alt(\G)$ are arc-transitive. The census contains no example $\G$ for which $\Alt(\G)$ would be half-arc-transitive but the quotient graph $\G_\B$ would be arc-transitive. We thus pose the following natural question.

\begin{question}
\label{que:HAT2}
Does there exist a tetravalent half-arc-transitive graph $\G$ such that $\Alt(\G)$ is half-arc-transitive but $\G_\B$ is arc-transitive?
\end{question}

The following result, which shows that the graph of alternating cycles is independent of taking quotients with respect to $\B$, can be of help when dealing with the above two questions. 

\begin{proposition}
 \label{prop:2}
Let $\G$ be a tetravalent $G$-half-arc-transitive graph for some $G \leq \aut(\G)$ such that $\att_G(\G)<\rad_G(\G)$. Then $\Alt_G(\G)\cong \Alt_{\overline{G}}(\G_{\B})$, where $\overline{G} = G/K_G(\G_\B)$. 
\end{proposition}

\begin{proof}  
Note that if $\G$ is loosely $G$-attached then $\G \cong \G_{\B}$, so there is nothing to prove. We can thus assume that $a=\att_G(\G) >1$. We now construct an isomorphism $\psi \colon \Alt_G(\G) \to \Alt_{\overline{G}}(\G_\B)$. For each $v \in V(\G)$ define $B_v$ to be the element of $\B$ containing $v$. Let $C=(u_0,u_1,\ldots,u_{2r-1}) \in V(\Alt_G(\G))$, where $r=\rad_G(\G)$, be a $G$-alternating cycle of $\G$ and $\L=2r/ a$. Suppose first that $a$ divides $r$. Then for each $0\leq i\leq \L-1$ we have $B_{u_i} = \{u_{i +j\L} \mid 0\leq j \leq a-1 \}$. In this case $(B_{u_0}, B_{u_1},B_{u_2},\ldots, B_{u_{\L-1}})$ is a $\overline{G}$-alternating cycle of $\G_\B$, which we set as the $\psi$-image of $C$. Suppose now that $a$ does not divide $r$ and recall that in this case $a$ is even. Then for each $0\leq i\leq 2\L-1$ we have $B_{u_i} = \{u_{i +j\L} \mid 0\leq j \leq a/2-1 \}$. In this case $(B_{u_0}, B_{u_1},B_{u_2},\ldots, B_{u_{2\L-1}})$ is a $\overline{G}$-alternating cycle of $\G_\B$, which we set as the $\psi$-image of $C$. It is easy to see that in both cases the mapping $\psi$ is a bijection. Since adjacency in the graph of alternating cycles is given by non-empty intersection, $\psi$ is clearly an isomorphism.
\end{proof}

Note that Proposition~\ref{prop:2} implies that in the case that the answer to Question~\ref{que:AT} is in the affirmative, the answer to Question~\ref{que:HAT2} is negative. Namely, suppose that the answer to Question~\ref{que:AT} is in the affirmative. If $\G$ is half-arc-transitive but $\G_\B$ is arc-transitive, then either $\G$ is tightly attached (in which case $\Alt(\G)$ is a cycle and is thus arc-transitive), or Proposition~\ref{prop:2} implies that $\Alt(\G)\cong \Alt_{\overline{G}}(\G_{\B})$, which, by assumption is arc-transitive.

There are several other curiosities one can observe by studying the census of all tetravalent half-arc-transitive graphs up to order $1000$ from~\cite{PotSpiVer15}. We mention just two of them. For all graphs $\G$ from the census such that $\Alt(\G)$ is also half-arc-transitive the graph $\G$ is either loosely attached or $\jum(\G) = 1$. Similarly, whenever $\G_\B$ is half-arc-transitive $|Q(\G)| = 1$.


\section{The graphs with $\boldsymbol{\att_G(\G) \nmid \rad_G(\G)}$}

Let $\G$ be a tetravalent graph admitting a half-arc-transitive group $G$ of automorphisms and let $r = \rad_G(\G)$ and $a = \att_G(\G)$. As we already pointed out in the Introduction it is well know that $a$ divides $2r$. However, it may happen that $a$ does not divide $r$ (see for instance~\cite{PotSpa17}, where the examples with $r = 3$ and $a = 2$ were characterized). Nevertheless, for all of the known examples where $a$ does not divide $r$, the graph $\G$ is in fact arc-transitive. This has led to the question, posed as Question 1 in~\cite{PotSpa17}, of whether in each tetravalent half-arc-transitive graph the attachment number divides the radius. In this section we use the alternating jump parameter and some of the results from the previous sections to address this question.

We first prove that in the case that $a$ does not divide $r$, the set $Q_G(\G)$ in fact consists of a single number, that is, the parameters $q_t$ and $q_h$ from Section~\ref{sec:jump} coincide. 

\begin{lemma}
\label{le:andivr}
Let $\G$ be a tetravalent $G$-half-arc-transitive graph for some $G \leq \aut(\G)$ and let $r = \rad_G(\G)$, $a = \att_G(\G)$ and $q = \jum_G(\G)$. If $a$ does not divide $r$ and $a \neq |V(\G)|$ then $q^2 \equiv \pm 1 \pmod{a}$, and so $Q_G(\G) = \{q\}$. 
\end{lemma}

\begin{proof}
Fix one of the two $G$-induced orientations of the edges of $\G$ and let $v \in V(\G)$. Let $q_t = q_t(v)$ and $q_h = q_h(v)$, so that $Q_G(\G) = \{q_t, q_h\}$ (confront Section~\ref{sec:jump}). Let $C=(u_0,u_1, \ldots, u_{2r-1})$ and $C'=(v_0,v_1, \ldots, v_{2r-1})$ be the two $G$-alternating cycles containing $v$, where $u_0 = v_0 = v$, $v$ is the tail of the two arcs of $C$, incident to it, and $u_{\L} = v_{q_t\L}$, where $\L = 2r/a$. Observe that, since $a \nmid r$ but $a \mid 2r$, the number $\L$ is odd. It follows that $u_{i\L}$ is the tail of the two arcs of $C$, incident to it, if and only if $i$ is even. In particular, $u_0$ is the tail and $u_{\L}$ is the head of the two arcs of $C$, incident to $u_0$ and $u_\L$, respectively. By Lemma~\ref{le:mult} we have that $u_{2\L}=v_{2q_t\L}$, and so $q_h(u_{\L}) = q_t$ (recall that $u_\L = v_{q_t\L}$ and $q_t \leq a/2$). Thus $q_h = q_t = q$, and so Lemma~\ref{le:inverse} implies $q^2 \equiv \pm 1 \pmod{a}$, as claimed. 
\end{proof}

The above lemma has the following useful corollary, describing the action of certain elements of the kernel $K_G(\Alt_G(\G))$.

\begin{corollary}
\label{cor:rotation}
Let $\G$ be a tetravalent $G$-half-arc-transitive graph for some $G \leq \aut(\G)$ such that $3 \leq a < r$, where $r = \rad_G(\G)$ and $a = \att_G(\G)$. Suppose $a$ does not divide $r$, let $q = \jum_G(\G)$, let $C$ be a $G$-alternating cycle of $\G$ and let $\gamma \in K_G(\Alt_G(\G))$ be a nontrivial element of the kernel of the action of $G$ on the graph $\Alt_G(\G)$. Then for some $0 < k \leq a/2$ the restriction of the action of $\gamma$ to $C$ is a $k\L$-step rotation, where $\L = 2r/a$. Moreover, for any two non-disjoint $G$-alternating cycles of $\G$ the restriction of the action of $\gamma$ to them is a $k\L$-step rotation on one of them and a $qk\L$-step rotation on the other. 
\end{corollary}

\begin{proof}
Note that the assumption $\gamma \in K_G(\Alt_G(\G))$ implies that $\gamma$ fixes $C$, as well as all the $G$-attachment sets containing the vertices of $C$, and so the existence of a suitable $0 < k \leq a/2$ is established in the proof of Lemma~\ref{le:rotation}. Moreover, that proof also shows that for each of the $G$-alternating cycles $C'$, having a non-empty intersection with $C$, the restriction of the action of $\gamma$ to $C'$ is a $qk\L$-step rotation. We can now repeat the same argument for each $C'$. Since $\gamma$ acts as a $qk\L$-step rotation on $C'$, it acts as a $k\L$-step rotation on each of the $G$-alternating cycles, having a non-empty intersection with $C'$ (recall that $q^2 \equiv \pm 1 \pmod{a}$ holds by Lemma~\ref{le:andivr}). By connectedness it now follows that for each pair of non-disjoint $G$-alternating cycles of $\G$ the restriction of the action of $\gamma$ on them is a $k\L$-step rotation on one of them and a $qk\L$-step rotation on the other one.
\end{proof}

We next show that in the case when $a$ does not divide $r$, the graph $\Alt_G(\G)$ is necessarily bipartite, unless possibly if $q = 1$ or $q = a/2-1$. 

\begin{lemma}
\label{le:bipartite}
Let $\G$ be a tetravalent $G$-half-arc-transitive graph for some $G \leq \aut(\G)$ such that $a \neq |V(\G)|$ and that $a$ does not divide $r$, where $r = \rad_G(\G)$ and $a = \att_G(\G)$. If $\jum_G(\G)\notin \{1, a/2 - 1\}$, then the graph $\Alt_G(\G)$ is bipartite. In other words, if $\Alt_G(\G)$ is not bipartite, then either $\jum_G(\G) = 1$ or  $\jum_G(\G) = a/2 - 1$ holds.
\end{lemma}

\begin{proof}
Denote $q = \jum_G(\G)$ and recall that Lemma~\ref{le:andivr} implies that $Q_G(\G) = \{q\}$. Suppose $q \notin \{1, a/2-1\}$. Then Lemma~\ref{le:inverse} implies that $a \geq 3$. Let $C = (u_0,u_1,\ldots , u_{2r-1})$ be a $G$-alternating cycle of $\G$ and let $\gamma \in G$ be the unique (confront Proposition~\ref{pro:stab}) element mapping $u_0$ to $u_{2\L}$ and $u_1$ to $u_{2\L+1}$, where $\L = 2r/a$ (observe that such a $\gamma$ does indeed exist since $2\L$ is even). Since $a \geq 3$ we have $2\L < 2r$, and so $\gamma$ is a nontrivial automorphism. It is clear that $\gamma$ is a $2\L$-step rotation of $C$, and so it fixes $C$ as well as all the $G$-attachment sets containing the vertices of $C$. By Lemma~\ref{le:rotation} we have $\gamma \in K_G(\Alt_G(\G))$, and so Corollary~\ref{cor:rotation} implies that for any pair of non-disjoint $G$-alternating cycles of $\G$ the restriction of the action of $\gamma$ to one of them is a $2\L$-step rotation while the restriction of its action to the other is a $2q\L$-step rotation. It now only remains to see that a $2\L$-step rotation of a cycle of length $2r$ is not a $2q\L$-step rotation of this cycle (in any of the two possible directions). If this was true, then one of $2\L \equiv 2q\L \pmod{2r}$ and $2\L \equiv -2q\L \pmod{2r}$ would have to hold, implying that one of $2\L(q-1)$ and $2\L(q+1)$ is divisible by $2r = a\L$. However, since by assumption $1 < q < a/2 - 1$ holds, none of these is possible. 
\end{proof}

We remark that tetravalent graphs $\G$, admitting a half-arc-transitive group of automorphisms $G \leq \aut(\G)$ such that $\att_G(\G)$ does not divide $\rad_G(\G)$ and $\Alt_G(\G)$ is not bipartite do exist. For instance, the results of~\cite{PotSpa17} show that any (non-bipartite) 2-arc-transitive cubic graph is the graph $\Alt_G(\G)$ for some tetravalent graph $\G$ admitting a half-arc-transitive group of automorphisms $G$ for which $\att_G(\G) = 2$ and $\rad_G(\G) = 3$, thus providing infinitely many examples $\G$ with the above mentioned situation. Of course, since $\att_G(\G) = 2$ we have $\jum_G(\G) = 1$ for all of these cases. We know of no example however, where $\jum_G(\G) = \att_G(\G)/2-1 > 1$ would hold, which thus leads us to the following question.

\begin{question}
\label{que:bip}
Does there exist a tetravalent graph $\G$, admitting a half-arc-transitive group of automorphisms $G \leq \aut(\G)$, such that $a \neq |V(\G)|$, $a \nmid r$ and $q = a/2-1 > 1$, where $a = \att_G(\G)$, $r = \rad_G(\G)$ and $q = \jum_G(\G)$, but the graph $\Alt_G(\G)$ is not bipartite?
\end{question}

We now give a result which is a considerable improvement of the results of~\cite{PotSpa17} towards the answer to Question~1 of~\cite{PotSpa17}. Moreover, if one is able to show that the answer to Question~\ref{que:bip} is negative, the next theorem will in fact almost completely solve Question~1 from~\cite{PotSpa17}, since the only remaining case will be the one when $\att_G(\G) = 4$.

\begin{theorem}
\label{the:andivr}
Let $\G$ be a tetravalent $G$-half-arc-transitive graph for some $G \leq \aut(\G)$ and let $r = \rad_G(\G)$, $a = \att_G(\G)$ and $q = \jum_G(\G)$. Suppose $a$ does not divide $r$ and $4 < a < r$. If $q = 1$ or the graph $\Alt_G(\G)$ is bipartite, then there exists an automorphism $\rho$ of $\G$, fixing all of the $G$-alternating cycles of $\G$ and acting as a $2r/a$-step rotation on at least one of them. Consequently, the graph $\G$ is arc-transitive. In particular, if $4 < a < r$, $a$ does not divide $r$ and $q \neq a/2-1$, then $\G$ is arc-transitive.
\end{theorem}

\begin{proof}
Fix one of the two $G$-induced orientations of the edges of $\G$. Since $a$ does not divide $r$, the number $\L = 2r/a$ is odd and  Lemma~\ref{le:andivr} implies that $Q_G(\G) = \{q\}$. Moreover, for each $G$-alternating cycle $C = (u_0, u_1, \ldots , u_{2r-1})$, where $u_0$ is the tail of the two arcs of $C$, incident with it, the vertices of the form $u_{2i\L}$ are the tails of the two arcs of $C$, incident with them, while the vertices of the form $u_{(2i+1)\L}$ are the heads of the two arcs of $C$, incident with them. Now, choose a $G$-alternating cycle $C$ and, as in the proof of Lemma~\ref{le:bipartite}, let $\gamma \in K_G(\Alt_G(\G))$ be such that its restriction to $C$ is a $2\L$-step rotation. By Corollary~\ref{cor:rotation} the restriction of the action of $\gamma$ to any two non-disjoint $G$-alternating cycles is a $2\L$-step rotation on one of them and a $2q\L$-step rotation on the other. In what follows we show that we can construct an automorphism $\rho$ of $\G$ such that $\rho^2 = \gamma$, having all the required properties.

To this end we first give two different labels to each vertex of $\G$. Before doing this we number the $G$-alternating cycles by denoting them with $C_1, C_2, \ldots , C_s$ and we choose a certain subset $I$ of the index set $S = \{1,2,\ldots, s\}$ in the following way. If $q = 1$ then set $I = S$. Suppose now that $q > 1$. By assumption $\Alt_G(\G)$ is bipartite in this case, and so the argument from the previous paragraph implies that the restriction of the action of $\gamma$ to the $G$-alternating cycles from one of the two sets of bipartition is a $2\L$-step rotation while its restriction to the $G$-alternating cycles from the other set of bipartition is a $2q\L$-step rotation. If $q \neq a/2-1$, then a $2\L$-step rotation is different from a $2q\L$-step rotation (in any of the two possible directions). In this case let $I$ be the set of all $i \in S$ for which the restriction of the action of $\gamma$ on $C_i$ is a $2\L$-step rotation. If however $q = a/2-1$ then simply choose one of the two sets of bipartition of $\Alt_G(\G)$ and let $I$ be the set of the indexes of the $G$-alternating cycles $C_i$ belonging to it. Observe that, in any case, for each $i \in I$ the restriction of the action of $\gamma$ to $C_i$ is a $2\L$-step rotation while for each $i \in S \setminus I$ the restriction of $\gamma$ to $C_i$ is a $2q\L$-step rotation. 

We now label the vertices of each $C_i$ by $v_j^i$, $j \in \ZZ_{2r}$, in such a way that $v_j^i$ and $v_{j+1}^i$ are consecutive vertices of $C_i$ for each $j$ and that $v_j^i\gamma = v_{j+2\L}^i$ for each $j \in \ZZ_{2r}$ whenever $i \in I$, while $v_j^i\gamma = v_{j+2q\L}^i$ for each $j \in \ZZ_{2r}$ whenever $i \in S \setminus I$. Observe that, since $q$ is coprime to $a$, each of $4\L \equiv 0 \pmod{2r}$ and $4q\L \equiv 0 \pmod{2r}$ contradicts $a > 4$, and so the above described labeling is unique up to cyclic rotations. Note however, that in this way each vertex received two different labels. 

We are now ready to define the mapping $\rho$, satisfying all of the properties from the statement of the theorem. We set
$$
	v_j^i\rho = \left\{\begin{array}{lcl}
		v_{j+\L}^i & ; & i \in I\\
		v_{j+q\L}^i & ; & i \in S \setminus I.
		\end{array}\right.
$$
To prove that $\rho$ is a well defined mapping let $v$ be a vertex of $\G$. Without loss of generality assume it belongs to the $G$-alternating cycles $C_1$ and $C_2$. We thus have $v = v_{j_1}^1 = v_{j_2}^2$ for some $j_1, j_2 \in \ZZ_{2r}$. We can further assume that we have $v_j^1\gamma = v_{j+2\L}^1$ and $v_j^2\gamma = v_{j+2q\L}^2$ for each $j \in \ZZ_{2r}$ (note that if $q = 1$ then $2q\L = 2\L$). Now, since $C_1$ and $C_2$ are non-disjoint and $Q_G(\G) = \{q\}$, it follows that $v_{j_1 + \L}^1 \in \{v_{j_2+q\L}^2, v_{j_2-q\L}^2\}$. If $v_{j_1 + \L}^1 = v_{j_2 - q\L}^2$, then Lemma~\ref{le:mult} implies that $v_{j_1 + 2\L}^1 = v_{j_2 - 2q\L}^2$. However, as $v_{j_1 + 2\L}^1 = v_{j_1}^1\gamma = v\gamma = v_{j_2}^2\gamma = v_{j_2 + 2q\L}^2$, this yields $v_{j_2 - 2q\L}^2 = v_{j_2 + 2q\L}^2$. But then $4q\L \equiv 0 \pmod{a\L}$, contradicting $a > 4$ (recall that $q$ is coprime to $a$). It thus follows that $v_{j_1+\L}^1 = v_{j_2 + q\L}^2$, and so $\rho$ is well defined, as claimed. That $\rho$ is indeed an automorphism of $\G$ is now clear from the definition. Moreover, it preserves each $G$-alternating cycle but does not respect the $G$-induced orientation of the edges of $\G$. It follows that $\G$ is arc-transitive. The last part of the theorem is now an immediate consequence of Lemma~\ref{le:bipartite}.
\end{proof}

Let us wrap up the paper with the following remark. If the remaining cases, not covered by Theorem~\ref{the:andivr}, can be taken care of to give an affirmative answer to Question 1 of~\cite{PotSpa17}, Theorem~\ref{the:quotient} will have further implications for a possible classification of all tetravalent half-arc-transitive graphs. Namely, since the tightly attached graphs have already been classified, it will remain to classify the tetravalent graphs admitting a loosely attached half-arc-transitive action and to determine how to construct all other tetravalent half-arc-transitive graphs as cyclic covers of them. The difficult problem of classifying all tetravalent graphs admitting a loosely attached half-arc-transitive action, which was proposed already by Wilson~\cite{Wil04}, is thus one of the central problems to be considered in future investigations on tetravalent half-arc-transitive graphs.

\end{document}